\documentclass[11.05pt,a4paper]{amsart}
\usepackage{rotating}
\usepackage[all]{xy}

\usepackage[utf8]{inputenc}
\usepackage{amscd,amssymb,amsopn,amsmath,amsthm,graphics,amsfonts,enumerate,verbatim,calc}
\usepackage{mathtools}

\usepackage{hyperref}
\usepackage[capitalise]{cleveref}

\usepackage{setspace}
\usepackage{mathpazo}
\usepackage{color}
\usepackage{latexsym}
\usepackage{amsthm,amsfonts,amssymb,mathrsfs}
\usepackage{rotating}
\usepackage[leqno]{amsmath}
\usepackage{xspace}
\usepackage[all]{xy}
\usepackage{longtable}
\headsep=5 mm \numberwithin{equation}{section}
\usepackage{mathpazo}
\usepackage{color}
\usepackage{latexsym}
\usepackage{amsthm,amsfonts,amssymb,mathrsfs}
\usepackage{rotating}
\usepackage[leqno]{amsmath}
\usepackage{xspace}
\usepackage[all]{xy}
\usepackage{longtable}
\headsep=5 mm \numberwithin{equation}{section}
\newcommand{\up}[1]{{{}^{#1}\!}}

\newtheorem{theorem}{Theorem}[section]
\newtheorem{lemma}[theorem]{Lemma}
\newtheorem{notation}[theorem]{Notation}
\newtheorem{proposition}[theorem]{Proposition}
\newtheorem{corollary}[theorem]{Corollary}

\newtheorem{problem}[theorem]{Problem}
\theoremstyle{definition}

\theoremstyle{remark}
\newtheorem{remark}[theorem]{Remark}
\newtheorem{fact}[theorem]{Fact}
\newtheorem{example}[theorem]{Example}
\newtheorem{observation}[theorem]{Observation}
\newtheorem{discussion}[theorem]{Discussion}
\newtheorem{question}[theorem]{Question}

\newcommand{\F}{\operatorname{F}}
\newcommand{\Ass}{\operatorname{Ass}}
\newcommand{\im}{\operatorname{im}}

\newcommand{\m}{\mathfrak m}
\newcommand{\n}{\mathfrak n}
\newcommand{\len}{\ell}
\newcommand{\tr}{\operatorname{tr}}

\newcommand{\PID}{\operatorname{PID}}

\newcommand{\Tr}{\operatorname{Tr}}

\newcommand{\pd}{\operatorname{pd}}

\newcommand{\ord}{\operatorname{ord}}\newcommand{\Q}{\operatorname{Q}}

\newcommand{\nil}{\operatorname{nil}}

\newcommand{\red}{\operatorname{red}}

\newcommand{\Ext}{\operatorname{Ext}}

\newcommand{\Tor}{\operatorname{Tor}}

\newcommand{\Hom}{\operatorname{Hom}}

\newcommand{\Ann}{\operatorname{Ann}}

\newcommand{\soc}{\operatorname{soc}}

\newcommand{\DVR}{\operatorname{DVR}}

\newcommand{\depth}{\operatorname{depth}}

\newcommand{\coker}{\operatorname{coker}}

\newcommand{\lo}{\longrightarrow}
\newcommand{\fm}{\frak{m}}
\newcommand{\fp}{\frak{p}}

\newcommand{\fn}{\frak{n}}

\begin{document}

\author[]{Mohsen Asgharzadeh}

\address{}
\email{mohsenasgharzadeh@gmail.com}

\title[rings of dimension one ]
{Integral closure of 1-dimensional rings}

\subjclass[2010]{ Primary 13G05; 13B22; 13A35.} 
\keywords{absolute integral closure; characteristic p methods; conductor ideal; integral closure; order; multiplicity; reflexivity; rings of dimenstion one}

\begin{abstract} 
	We study certain properties of modules over 1-dimensional local integral domains. First, we examine the order of the conductor ideal and its expected relationship with multiplicity. Next, we investigate the reflexivity of certain colength-two ideals. Finally, we consider the freeness problem of the absolute integral closure of a $\DVR$, and connect this to the  reflexivity problem of
	$R^{\frac{1}{p^n}}$. 
\end{abstract}

\maketitle
\tableofcontents

\section{Introduction}
Let $(R, \mathfrak{m}, k)$ be a commutative Noetherian local ring, and let $\overline{R}$ denote the normalization of $R$ in its total ring of fractions. The \textit{conductor} of $R$ is defined as  
$
\mathfrak{C}_R := \Ann_R\left(\frac{\overline{R}}{R}\right)
$. 
For an ideal $I$ of $R$, the \textit{order} of $I$ with respect to $\mathfrak{m}$ is given by  
$
\ord(I) := \sup \{ n \in \mathbb{Z}_{\geq 0} \mid I \subseteq \mathfrak{m}^n \}.
$ Here, $e_\mathfrak{m}(R)$ denotes the Hilbert-Samuel \textit{multiplicity} of $R$, and $\mu(\mathfrak{m})$ represents the minimal number of generators of $\mathfrak{m}$.
The following questions, originally posed in \cite[Questions 5.4 + 5.6]{dey}, explore the behavior of the conductor ideal in this context.

\begin{question}\label{q}
	\begin{enumerate}
		\item[(i)]  Let $R = k[[H]]$ be a numerical semigroup ring with $e_\fm(R) - \mu(\fm) = 1$. Then, is $\ord(\mathfrak{C}) = 2$?
		\item[(ii)] Let $R$ be a complete local hypersurface domain of dimension 1. Then, is $\ord(\mathfrak{C}) = e_\fm(R) - 1$?
	\end{enumerate}
\end{question}

Regarding Question \ref{q}(i), Dey and Dutta proved that it holds when $R = k[[t^4, t^a, t^b]]$, where $4 < a$ (even) $< b$ (odd).  
We show that Question \ref{q}(i) does not hold in the (non-)almost Gorenstein case. As an example, we present:

\begin{example}\label{inex}
	Let $R := k[[t^4, t^5, t^7]]$. Then $e_\fm(R) - \mu(\fm) = 1$, but $\ord(\mathfrak{C}) = 1$.
\end{example}
We extend Example \ref{inex} to a more general setting. Additionally, following several advances in commuting conductors, we provide a more detailed analysis of Question \ref{q}. Namely, suppose that $(R,\fm)$ is a complete domain of positive dimension with $\mathfrak{C} + xR = \fm$ for some $x \in \fm$ (not necessarily 1-dimensional). We show that $\ord(\mathfrak{C}) = 1$.  The latter motivates us to explore whether Question \ref{q}(i) remains valid over any hypersurface domain. Moreover, it holds over Gorenstein rings with a normal maximal ideal, as shown by Ooishi \cite{Oo}. We further demonstrate that Question \ref{q}(i) is valid over Gorenstein rings with non-node curve singularities. The examples from Section 2 motivate the search for 
	determining the set
$
	\big\{a \geq 7 : \ord(\mathfrak{C}_{k[[t^4,t^5,t^a]]}) = 1\big\}.
$

Regarding Question \ref{q}(ii), it was established in \cite{or}  that it holds over Gorenstein domains with 1-step normalization. We recover this result. Additionally, we prove its validity in the following new cases:

\begin{observation}
	\begin{enumerate}
		\item[(i)] Suppose $R$ is a localization of an affine hypersurface over a field of characteristic zero and of dimension 1. If $\mathfrak{C}= \fm^n$, then $\ord(\mathfrak{C}) = e_\fm(R) - 1$.
		\item[(ii)] Let $R$ be a complete local hypersurface domain over $\mathbb{C}$ with $e(R) \leq 3$ and dimension 1. Then $\ord(\mathfrak{C}) = e_\fm(R) - 1$.
		\item[(iii)] Let $(R,\fm)$ be a 1-dimensional essentially affine hypersurface over a field of characteristic zero. Then $\ord(\mathfrak{C}) \leq e_\fm(R) - 1$.
	\end{enumerate}    
\end{observation}Then we show:
\begin{theorem}(See Theorem \ref{thm:conductor}).
Let $(R,\m,k)$ be a complete local hypersurface domain of dimension
one. Then
$\ord(\mathfrak{C})=e_{\m}(R)-1.
$
\end{theorem}
The following question was asked in \cite[Question 7.17]{dao}:
\begin{question}\label{q1}
	Let $R$ be a Cohen-Macaulay local ring of dimension one. When is an ideal of colength two a trace ideal? When is it reflexive?
\end{question}
Regarding Question \ref{q1}, we present an ideal of colength two in a one-dimensional domain that is not reflexive, while all ideals of colength one are reflexive. We also observe:

\begin{observation}
	Let $(R,\fm)$ be a 1-dimensional complete domain of minimal multiplicity. The following hold:
\begin{enumerate}
\item[(i)] Suppose $R$ is almost Gorenstein, but not a hypersurface. Then the sequence
$0 \to \omega_R \to \omega_R^{\ast\ast} \to k \to 0
$ 	is exact, and in particular, $\omega_R^{\ast\ast} \cong \fm$.
\item[(ii)] Suppose any colength-two ideal is reflexive and trace. Then $R$ is a hypersurface.
\item[(iii)] There exists an almost Gorenstein domain of minimal multiplicity, but not far-flung Gorenstein, where every colength-two trace ideal is reflexive.
\item[(iv)] Suppose there exists a self-dual ideal of colength two that is $q$-torsionless with $q > 2$. Then $R$ is a hypersurface.
	\end{enumerate} 
\end{observation}
This leads to the following conclusion:

\begin{corollary}
	Let $(R,\fm)$ be a 1-dimensional domain of minimal multiplicity. Suppose there exists a totally reflexive ideal  of colength two. Then $R$ is a hypersurface.
\end{corollary}
The absolute integral closure of an integral domain \( R \), denoted by \( R^+ \). The following questions were raised in \cite[Question 3.19]{p} and its subsequent paragraph:  

\begin{question}\label{q4}  
	Let \( R \) be a (regular) complete local domain of positive characteristic.  
 What is the projective dimension of \( R^+ \) as an \( R \)-module?  
Is it zero?  
\end{question}  

In order to have a motivation for the freeness problem, recall that an excellent local domain \( R \) of positive characteristic is regular if and only if \( R^+ \) is flat over \( R \). In particular, if \( R \) is regular, then \( R^+ \) is flat over \( R \). However, since \( R^+ \) is not of finite type, the standard techniques from commutative algebra do not immediately establish that \( R^+ \) is free.  
In the final section, we construct a 1-dimensional regular ring \( R \) for which \( \pd_R(R^+) = \dim R > 0 \)\footnote{It may be wonder if we look at \cite[Ex. 8.5]{Hut} which claims $R^{+gr}$ is free if $R$ is standard graded integral domain over a filed of prime characteristic.}. However, the above ring is not complete. We show if the above conclusion  is failed for its completion, then $R^{\frac{1}{p^n}}$ is not reflexive.
This is the subject of the second part of Section 5, where we deal with the reflexivity problem of
 $R^{\frac{1}{p^n}}$ from \cite[Question 7.19]{dao}.

\section{When is order  of conductor  2?}

In this section we study the following question:
\begin{question}
Let	$R=k[[H]]$ be a numerical semigroup ring with $e_\fm(R)-\mu(\fm)=1$. Then is $\ord(\mathfrak{C})=2$?
\end{question}
The \emph{trace} of an $R$-module $M$, denoted $\tr_R(M)$, is defined to be
 $
\tr_R(M)=\sum_{f\in M^*}\im(f),
$ with the convenience that $M^\ast:=\Hom_R(M,R)$. Recall from \cite{tiny} that
a ring $R$ is \emph{far-flung Gorenstein} provided 
$
\mathfrak{C} = \operatorname{tr}(\omega_R),
$
where $\omega_R$ is the canonical module.

\begin{lemma}\label{t}
(See \cite[Corollary 6.3]{tiny}). Let $a\geq  3$ and d be coprime nonnegative integers and $H = \langle a,a+d, \cdots, a +(a -1)d\rangle$. Then $R := k[[H]]$ is a far-flung Gorenstein ring if and only if $d =1$.
\end{lemma}
Recall from \cite{HHS} that
a ring $R$ is \emph{nearly Gorenstein} provided 
$
\fm\subseteq\operatorname{tr}(\omega_R).
$
Also,   
 $R$ is called \emph{almost Gorenstein} if there exists an exact sequence
 $0 \to R \to\omega_R\to C \to 0$ with
$e_{\fm}(C)=\mu(C)
$. For more details, 
 see \cite{Gt}. Also, we cite \cite{ab} as a reference  monograph  on totally reflexive modules.
\begin{lemma}\label{ko}
	(See \cite[Corollary 3.7]{kob}). Let $R = k[[t^4, t^5, t^7]]$. Then $R$ is an almost Gorenstein local ring of dimension
	one. Also, $R$ is nearly Gorenstein and any totally reflexive module is free.
\end{lemma}
We cite \cite{ab} to see  definition and basic properties of  totally reflexive modules. Now, we are ready to solve the first question:

\begin{example}
	Let $R:=k[[t^4,t^{5},t^{7}]]$. Then $e_\fm(R)-\mu(\fm)=1$ but $\ord(\mathfrak{C})=1$.
\end{example}

\begin{proof}
We know from \cite[Ex. 11.9]{HS} that  
$
e_\mathfrak{m}(R) = \min\{4,5,7\}
$
and also that $\mu(\mathfrak{m}) = 3$. Thus,  
$
e_\mathfrak{m}(R) - \mu(\mathfrak{m}) = 1.
$
Recall  that  
\[
H = \langle a, a + d, \dots, a + (a - 1)d \rangle
\]
with $a := 4$ and $d = 1$.  
Thanks to Lemma \ref{t}, the ring $R$ is  far-flung Gorenstein. Hence,   
$
\mathfrak{C} = \operatorname{tr}(\omega_R) 
$.  
In view of Lemma \ref{ko}, $R$ is nearly Gorenstein. Then, we have  
$
\mathfrak{m} \subseteq \operatorname{tr}(\omega_R).
$
However, $R$ is  non-regular  but totally reflexive module is free. It follows that $R$ is not Gorenstein, and consequently,  
$
\operatorname{tr}(\omega_R) \neq R.
$
Therefore,  
$
\mathfrak{m} = \operatorname{tr}(\omega_R) = \mathfrak{C}.
$
By definition,  
$
\operatorname{ord}(\mathfrak{C}) = 1.
$
\end{proof}

Here, is an example which is not almost Gorenstein:
\begin{example}
	Let $R:=k[[t^6,t^{8},t^{11},t^{13},t^{15}]]$. Then $e_\fm(R)-\mu(\fm)=1$ but $\ord(\mathfrak{C})=1$. 
\end{example}

\begin{proof}We know  $e_\fm(R)=6$ and also $\mu(\fm)=5$. Hence, $e_\fm(R)-\mu(\fm)=1$. Recall that $t:=\frac{t^6\times t^6}{t^{11}}$ is integral over $R$. 
	So, $\overline{R}=k[[t]]$. Since
	$t^{11}k[[t]]\subseteq R$ we see 
	$t^{11}\in \mathfrak{C}$. Since
$t^{11}\in\fm \setminus \fm^2$
we deduce $\ord(\mathfrak{C})=1$. In view of \cite[7.6]{dao}, $\fm^2$ is not reflexive. According to \cite[7.3]{dao}, the ring is not almost Gorenstein.
\end{proof}

\begin{proposition}
Let	$(R,\fm)$ be any  complete domain with  $\mathfrak{C}+xR=\fm$ for some $x\in \fm$. Then	$\ord(\mathfrak{C})=1.$
\end{proposition}

\begin{proof}
Suppose, for the sake of contradiction, that $\mathfrak{C} \subseteq \mathfrak{m}^2$. Since $\mathfrak{C} + xR = \mathfrak{m}$ and $\mathfrak{C} \subseteq \mathfrak{m}^2$, it follows that  
$
\mathfrak{m}/\mathfrak{C} = (\mathfrak{C} + xR)/\mathfrak{C}
$
is cyclic. The natural surjection  
$
\mathfrak{m}/\mathfrak{C} \to \mathfrak{m}/\mathfrak{m}^2 \to 0
$
implies that $\mathfrak{m}/\mathfrak{m}^2$ is cyclic as well. Consequently, we obtain $\mu(\mathfrak{m}) = 1$, which in turn implies that $R$ is a discrete valuation ring ($\DVR$).  
Since $R$ is a $\DVR$, we must have $\mathfrak{C} = R$. However, this leads to a contradiction because  
$
R = \mathfrak{C} + xR = \mathfrak{m},
$
whereas $\mathfrak{m}$ is a proper ideal of $R$. Thus, our assumption that $\mathfrak{C} \subseteq \mathfrak{m}^2$ must be false.  
We have therefore shown that $\operatorname{ord}(\mathfrak{C}) \leq 1$. Moreover, note that $\operatorname{ord}(\mathfrak{C}) = 0$ if and only if $\mathfrak{C} = R$, which we have already ruled out. Hence, it follows that  
$
\operatorname{ord}(\mathfrak{C}) = 1.
$  
\end{proof}
However, Question 2.1 is true for $1$-dimensional hypersurfaces.

\begin{corollary}
	Let	$R$ be a $1$-dimensional complete local hypersurface domain with $e_\fm(R)-\mu(\fm)=1$. Then  $\ord(\mathfrak{C})=2$. 
\end{corollary}

\begin{proof}
Since $\mu(\fm)=2$ it follows that
$e_\fm(R)=3$. In view of Proposition \ref{e3} (see below) we observe that $\ord(\mathfrak{C})=e_\fm(R)-1=3-1=2$, as claimed.
\end{proof}
In order to extend the previous result to the realm of Gorenstein rings, we impose some strong restrictions. For instance, the reducedness of $\textit{gr}_\fm(R)$ guarantees that $\fm$ is normal (see \cite[Ex. 5.7]{HS}).

\begin{fact} 
	Let $(R,\fm)$ be a $1$-dimensional Gorenstein domain with $e_\fm(R) - \mu(\fm) = 1$.
\begin{enumerate}
\item[i)] If $\fm$ is normal, then $\ord(\mathfrak{C}) = 2$. The following holds:
\item[ii)] Suppose $R$ is localization of a singular curve at not-node point. Then $\ord(\mathfrak{C}) = 2$.
	\end{enumerate}
\end{fact}

\begin{proof}
i)  This  is in \cite[12(2)]{Oo}.

ii) By \cite[final claim]{or}, we deduce that $\mathfrak{C} \subseteq \fm^2$, and hence $\ord(\mathfrak{C}) \geq 2$. Furthermore, by \cite[5.1]{dey}, we have $\ord(\mathfrak{C}) \leq 2$, completing the proof.
\end{proof}
\begin{example}\label{oor}
Let $R_1 := k[[t^e, \ldots, t^{2e-2}]]$ with $e \geq 3$ and $R_2 := k[[t^4, t^5, t^{11}]]$.	Then $e_\fm(R_i) - \mu(\fm_i) = 1$,  $\ord(\mathfrak{C}_{R_1}) = 2$, while  $\ord(\mathfrak{C}_{R_2}) = 1$.
\end{example}

\begin{proof}
The claim for $R_1$ is in \cite[5.9]{Oo1}. In fact, $\mathfrak{C} = \fm^2$.
Let $R:=R_2$. It is easy to see that $t^{11} \overline{R} \subseteq R$. Since $t^{11} \in\fm\setminus\fm^2 $, we observe $\ord(\mathfrak{C}) = 1$. For more details, see \cite[1.4]{or}.
\end{proof}
\iffalse
The examples suggested in this section motivate the search for the following set:

\begin{problem}\label{cp}
	Determine the set
	\[
	\big\{a \geq 7 : \ord(\mathfrak{C}_{k[[t^4,t^5,t^{a}]]}) = 1\big\}.
	\]
\end{problem}\fi
\section{The order  of conductor looking for multiplicity }
The results of this section are in support of:
\begin{question}
 Let $(R,\fm)$ be a complete  hypersurface domain of dimension $1$. Is $\ord(\mathfrak{C})=e_\fm(R)-1$? 	
\end{question}
\begin{observation}
Let $(R,\fm,\mathbb{C})$ be a complete local domain of dimension $1$ and multiplicity 2.  Then   $\ord(\mathfrak{C})=e_\fm(R)-1$.
	\end{observation}	

\subsection{Presenting a bound}
Since the above ring is Cohen-Macaulay
and of multiplicity 2, it should be hypersurface. So, it is the special case of:
\begin{proposition}\label{e3}
Let $(R,\fm,\mathbb{C})$ be a complete local hypersurface domain of dimension $1$ and multiplicity $\leq 3$. Then   $\ord(\mathfrak{C})=e_\fm(R)-1$.	\end{proposition}

\begin{proof}
Note that $\ord(\mathfrak{C})=0$ if and only if $\mathfrak{C}=R$, if and only if $R$ is normal, if and only if $R$ is regular, if and only if $e_\mathfrak{m}(R)=1$.  

Now, assume $e_\mathfrak{m}(R)=2$.  
Since $R$ is analytically unramified, it follows from a theorem of Buchweitz-Greuel-Schreyer that $R$ is a simple singularity. Moreover, by a theorem of Greuel-Knörrer, since $R$ is a hypersurface, it must have an $A_n$-singularity. In other words,  
$
R := \mathbb{C}[[X,Y]]/(X^2+Y^n).
$ 
Suppose, for the sake of contradiction, that $n=2m$ is even. Then,  
\(
0 = X^2 + Y^n = (X + iY^m)(X - iY^m),
\)
where $i \in \mathbb{C}$ is the imaginary unit satisfying $i^2 = -1$. This contradicts the assumption that $R$ is a domain. Hence, $n$ must be odd. From this, we conclude  
$
R \cong \mathbb{C}[[t^2,t^n]].
$  
Since $\gcd(2,n) = 1$, it follows from \cite[5.5]{dey} that  
$
\ord(\mathfrak{C}) = e_\mathfrak{m}(R) - 1.
$ 

Next, assume $e_\mathfrak{m}(R)=3$. By \cite[Theorem 8.5]{y}, we see that $R$ must be one of the following:  
\begin{center}
 $\big\{  \mathbb{C}[[X,Y]]/(X^3+Y^4),  
 \mathbb{C}[[X,Y]]/(X^3+Y^5),  \mathbb{C}[[X,Y]]/(X^3+XY^3)\big\}$.  
\end{center}  

In the third case, we derive a contradiction, as the ring contains a zero-divisor. Thus, we must have either  
$R \cong \mathbb{C}[[t^3,t^4]]$  or $ R \cong \mathbb{C}[[t^3,t^5]].
$  
Since $\gcd(3,4) = 1 = \gcd(3,5)$, it follows from \cite[5.5]{dey}, or even directly\footnote{$t^8 \in \mathfrak{C}_{R_1} \setminus \mathfrak{m}^3$ and $t^7 \in \mathfrak{C}_{R_2} \setminus \mathfrak{m}^3$. Thus, $\ord(\mathfrak{C}) = 2$.}, that  
$
\ord(\mathfrak{C}) = e_\mathfrak{m}(R) - 1.
$
\end{proof}

For any $f\in R$, we define  $\ord(f):= \ord(fR)$.

\begin{proposition}\label{in}
Suppose $R$ is an affine hypersurface over a field of zero characteristic and   of dimension $1$. Then $\ord(\mathfrak{C})\leq e_\fm(R)-1$.
\end{proposition}

\begin{proof}
Suppose $R= \mathbb{C}[X,Y]/(f)$ for some $f\in \mathbb{C}[X,Y]$. Recall that  
$
e_\mathfrak{m}(R) = \ord(f)
$  
(see \cite[Example 11.2.8]{HS}).  
Suppose, for the sake of contradiction, that  
$
\ord(\mathfrak{C}) > e_\mathfrak{m}(R) - 1.
$ 
Then we have  
$
\mathfrak{C} \subseteq \mathfrak{m}^{e_\mathfrak{m}(R)}.
$ 
It is known from \cite{vas}, with a contribution to E.~Noether, that  
\[
\left( \frac{\partial f}{\partial X}, \frac{\partial f}{\partial Y} \right) \subseteq \mathfrak{C}.
\]  
By symmetry, we may assume that $\frac{\partial f}{\partial X}$ is nonzero and has order exactly $\ord(f) - 1$. Combining these facts, we conclude that  
$
\frac{\partial f}{\partial X} \in \mathfrak{m}^{\ord(f)}.
$ 
This leads to a contradiction. 
\end{proof}	
	
	\begin{remark}
Let $(R,\fm)$ be a complete hypersurfacee local domain of dimension $1$ and  $\mathfrak{C}=\fm^i$ with $i\leq 3$, then $\ord(\mathfrak{C}) \geq e_\fm(R)-1$.  	\end{remark}

\begin{proof}
We present the case for $i = 3$. By the exact sequence  
$
0 \to \mathfrak{m}/\mathfrak{m}^2 \to R/\mathfrak{m}^2 \to R/\mathfrak{m} \to 0,
$  
we know that  
$
\ell(R/\mathfrak{m}^2) = \mu(\mathfrak{m}) + 1 = 3.
$  
Since $\mu(\mathfrak{m}) = 2$, we deduce that $\mu(\mathfrak{m}^2) \leq 3$. Similarly,  
\[
\ell(R/\mathfrak{m}^3) = \mu(\mathfrak{m}^2) + \ell(R/\mathfrak{m}^2) \leq 3 + \ell(R/\mathfrak{m}^2) = 6.
\]  
Recall from \cite[Corollary 12.2.4]{HS} that  
\[
12 \geq 2\ell(R/\mathfrak{m}^3) = 2\ell(R/\mathfrak{C}) = \ell(\overline{R}/\mathfrak{C}) =: \ell,
\]  
where $\overline{R}$ is the integral closure of $R$. Since $R$ is complete-local, $\overline{R}$ is also local and, moreover, regular. Thus, its maximal ideal is principal, and we denote its uniformizer by $t$. Then  
$
\mathfrak{m}^3 \overline{R} = t^\ell \overline{R}.
$ 
Since the extension of ideals is multiplicative, we deduce that  
$
\mathfrak{m}^{3n} \overline{R} = t^{n\ell} \overline{R}.
$  
Now, recall that  
\[
e_\mathfrak{m}(\overline{R}) = \lim_{n\to \infty} \frac{\ell(\overline{R}/\mathfrak{m}^{3n} \overline{R})}{3n}  
= \lim_{n\to \infty} \frac{\ell(\overline{R}/t^{n\ell} \overline{R})}{3n}  
= \frac{n\ell}{3n} \leq \frac{12}{3} = 4.
\]  
Consequently,  
$
e_\mathfrak{m}(R) = \frac{e_\mathfrak{m}(\overline{R})}{\operatorname{rank} \overline{R}} \leq 4 = \ord(\mathfrak{m}^{3n}) + 1 = \ord(\mathfrak{C}) + 1.
$ 
	\end{proof}

The previous remark illustrates a more general phenomena:
\begin{proposition}\label{power}
Let $(R,\fm)$ be a hypersurfacee local domain of dimension $1$ and  $\mathfrak{C}=\fm^{n-1}$. Suppose
$\ord(\mathfrak{C})\leq e_\fm(R)-1$.
Then $\ord(\mathfrak{C})= e_\fm(R)-1$.  
	\end{proposition}
\begin{proof} We will use Matlis' step lemma:
\[
\mu(\mathfrak{m}^i) = \begin{cases}
		i + 1 & \text{if } i \leq e_\mathfrak{m}(R) - 1, \\
		e_\mathfrak{m}(R) & \text{otherwise,}
		\end{cases}
\]
(see \cite[Theorem 13.7]{E}).  
By the exact sequence  
	\(	0 \to \mathfrak{m}^{n-2}/\mathfrak{m}^{n-1} \to R/\mathfrak{m}^{n-1} \to R/\mathfrak{m}^{n-2} \to 0,
		\)  we know that \(
\ell(R/\mathfrak{m}^{n-1}) = \mu(\mathfrak{m}^{n-2}) + \ell(R/\mathfrak{m}^{n-2}).
		\) 
Thanks to induction, we have  
\[
		\begin{array}{ll}
		\ell(R/\mathfrak{m}^{n-1}) &= \mu(\mathfrak{m}^{n-2}) + \mu(\mathfrak{m}^{n-3}) + \cdots + \mu(\mathfrak{m}) + \ell(R/\mathfrak{m}) \\
		&= (n-1) + (n-2) + \cdots + 2 + 1 \\
		&= \frac{n(n-1)}{2}.
		\end{array}
\]  
		
		It turns out from the Gorenstein property that  
	$
		\ell(\overline{R}/\mathfrak{m}^{n-1} \overline{R}) = 2 \ell(R/\mathfrak{m}^{n-1}) = n(n-1),
	$	and so  
$
\mathfrak{m}^{k(n-1)} \overline{R} = t^{k n(n-1)} \overline{R}$  for any $ k \in \mathbb{N}.
$	Consequently,  
$e_\mathfrak{m}(R) = e_\mathfrak{m}(\overline{R}) = \frac{k n(n-1)}{k (n-1)} =n $
		and the claim follows.\end{proof}

Concerning Proposition \ref{power}, the hypersurface assumption is important. For example,
the conductor of $k[[t^4,t^5,t^6]]$ is $\fm^2$ (see either Example \ref{oor} or \cite[7.13]{dao}). Since the ring is complete-intersection, it is enough to apply:

	\begin{proposition}\label{37}
	Let $(R,\fm)=k[[X,Y,Z]]/ I$ with $I\subseteq \fn^2$ be a Gorenstein local domain of dimension $1$ and  $\mathfrak{C}=\fm^{2}$. 
	Then $\ord(\mathfrak{C})= e_\fm(R)-\mu(I)= e_\fm(R)-2$.	\end{proposition}
\begin{proof}  By the exact sequence  
$
	0 \to \mathfrak{m}/\mathfrak{m}^2 \to R/\mathfrak{m}^2 \to R/\mathfrak{m} \to 0,
$  
	we know that  
	\(
	\ell(R/\mathfrak{m}^2) = \mu(\mathfrak{m}) + \ell(R/\mathfrak{m}).
	\) 
	It turns out that  
$
	\ell(\overline{R}/\mathfrak{m}^2 \overline{R}) = 2 \ell(R/\mathfrak{m}^2) = 8,
$ 
	and so  
$
	\mathfrak{m}^{2n} \overline{R} = t^{8n} \overline{R}$   {for any } $n \in \mathbb{N}.$
Consequently,  
	\(
	e_\mathfrak{m}(R) = e_\mathfrak{m}(\overline{R}) = \frac{8n}{2n} = 4.
	\) 
In particular,  
$
	\ord(\mathfrak{C}) = e_\mathfrak{m}(R) - 2.
$ 
	The ideal $I$ is prime, because $R$ is a domain. Since $R$ is Gorenstein, $I$ is prime and $\operatorname{ht}(I) = 2$, we know that $I$ is generated by a regular sequence of length two. In particular, $R$ is a complete intersection.
\end{proof}

\subsection{A length argument for the first blow-up}
We first recall the following fact:

\begin{lemma}\label{lem:duality}(Matlis)
	Let $(A,\n,k)$ be Artinian local, and let $D(-)=\Hom_A(-,E_A(k))$.
	On finite modules, $D$ is exact, preserves length, and the evaluation
	$N\to D(D(N))$ is an isomorphism. Moreover,
	\[
	\len(E_A(k))=\len(A),\qquad D(E_A(k))\cong A.
	\]
	If $\dim_k\soc(A)=1$, then $E_A(k)\cong A$. Consequently, for every
	ideal $L\subseteq A$,
$$(\ast):\quad
		\len(A/\Ann_A L)=\len(L),\qquad
		\Ann_A\Ann_A L=L.
$$
\end{lemma}
\begin{proof}
Injectivity of $E_A(k)$ gives exactness. Its socle is $k$, so $D(k)\cong k$; a composition series then proves preservation of length and finiteness.
	In particular $D(A)=E_A(k)$ has length $\len(A)$.
Maps to $E_A(k)$ separate points of a finite module.
	If $0\ne z\in N$, then $z$ generates $Az$, and its image in
	$Az/\n Az$ is nonzero by Nakayama's lemma. A map from this quotient to
	$k\subseteq E_A(k)$ nonzero on $z$ extends to $N$ by injectivity.
	Thus evaluation is injective, and equality of lengths makes it an
	isomorphism. Taking $N=A$ gives $D(E_A(k))\cong A$.
Every nonzero submodule of an Artinian module meets its socle. When
	$\soc(A)\cong k$, extend an isomorphism of this socle to the socle of
	$E_A(k)$ to a map $A\to E_A(k)$. Its kernel is zero, and equal lengths
	make it an isomorphism. Hence $D$ may be identified with $\Hom_A(-,A)$.
	Now $D(A/L)=\Ann_A L$, proving the first equality in
$(\ast)$. Applying it twice and using
	$L\subseteq\Ann_A\Ann_A L$ proves the second.
\end{proof}

\begin{lemma}\label{lem:blowup}
	Let $(R,\m)$ be a one-dimensional Gorenstein local domain with finite
	normalization. Suppose a nonzero $x\in\m$ and an integer $r\ge1$
	satisfy
$$(\ast)	\quad 
		\m^{r+1}=x\m^r,\qquad
		2\len(R/\m^r)=\len(R/x^rR).
$$
	Set $T:=x^{-r}\m^r\subseteq\Q(R)$. Then $T$ is a finite birational
	ring, and
$R:T=\m^r=x^rT.$
	Writing $\mathfrak{D}=T:\overline R$, one has
$\mathfrak{C}=x^r\mathfrak{D}\subseteq\m^r,$ 
and $\mathfrak{C}\not\subseteq\m^{r+1}.
 $
	In particular $\ord(\mathfrak{C})=r$.
\end{lemma}
\begin{proof}
	Induction from the first hypothesis gives
	$\m^{r+j}=x^j\m^r$ for every $j\ge0$. In particular
	$\m^{2r}=x^r\m^r$. Since $x^r\in\m^r$, the finite module
	$T=x^{-r}\m^r$ contains $R$ and is closed under multiplication.
	It is finite and birational, hence integral over $R$ and contained
	in $\overline R$.
Let us denote $J=R:T=(x^rR):_R\m^r$. The equality of powers shows
	$\m^r\subseteq J$. The Artinian ring $A=R/x^rR$ is Gorenstein.
	Let $L=\m^r/x^rR$. Then $J/x^rR=\Ann_A L$, so
	Lemma~\ref{lem:duality} gives
$$\len(R/J)=\len(L)
		=\len(R/x^rR)-\len(R/\m^r)
	=\len(R/\m^r).
$$
	Together with $\m^r\subseteq J$, equality of colengths proves
	$J=\m^r=x^rT$.
Recall that $\mathfrak{C}=R:\overline R$. If $c\in \mathfrak{C}$ and $b\in\overline R$, then
	$cbT\subseteq c\overline R\subseteq R$, hence $cb\in J=x^rT$.
	Thus $(c/x^r)\overline R\subseteq T$, or $c\in x^r\mathfrak{D}$.
	Conversely, if $d\in \mathfrak{D}$, then
	$x^rd\overline R\subseteq x^rT=J\subseteq R$. This proves
	$\mathfrak{C}=x^r\mathfrak{D}$ and $\mathfrak{C}\subseteq\m^r$.
	Suppose $\mathfrak{C}\subseteq\m^{r+1}=x^{r+1}T$. Cancelling $x^r$ gives
	$\mathfrak{D}\subseteq xT$. The conductor $\mathfrak{D}$ is an ideal of $\overline R$,
	so $x^{-1}\mathfrak{D}$ is an $\overline R$-stable submodule contained in $T$.
	Every such submodule is contained in $T:\overline R=\mathfrak{D}$. To see this,
	if $z\overline R\subseteq T$, then $z\in \mathfrak{D}$. Consequently
	$x^{-1}\mathfrak{D}\subseteq \mathfrak{D}$, and $\mathfrak{D}\subseteq x\mathfrak{D}$. The reverse inclusion
	is automatic, so $\mathfrak{D}=x\mathfrak{D}$. The nonzero conductor $\mathfrak{D}$ is finite over
	$R$, and $x\in\m$; Nakayama's lemma gives the contradiction $\mathfrak{D}=0$.
\end{proof}

\begin{theorem}\label{thm:conductor}
	Let $(R,\m,k)$ be a complete  hypersurface domain of dimension
	one. Then
$\ord(\mathfrak{C})=e_{\m}(R)-1.
$
\end{theorem}
\begin{proof}
	A complete local domain is excellent, so its normalization is finite;
	this is the standard finiteness-of-normalization consequence of
	excellence. If $R$ is regular, it is a discrete valuation ring, its
	conductor is $R$, and the formula is $0=1-1$. Assume henceforth
	that $R$ is singular. Its hypersurface presentation is
	$R=S/(f)$, where $(S,\n,k)$ is regular local of dimension two,
	$0\ne f\in\n^2$, and $R$ is Gorenstein.
	
	First suppose $k$ is infinite. With a regular system of parameters
	of $S$, its associated graded ring is $k[X,Y]$. Write
	$e=\ord_{\n}(f)$. Since the associated graded ring of $S$ is a
	domain, orders of products add, and the initial ideal of $(f)$ is
	generated by the initial form $f_e$. Hence
$$
(\ast):	\quad	\operatorname{gr}_{\m}R\cong k[X,Y]/(f_e),\qquad
		\dim_k\m^j/\m^{j+1}=\min\{j+1,e\}.
$$
	In particular $e=e_{\m}(R)$.
	Choose a linear form $l$ not dividing $f_e$ and lift it to $x\in\m$.
	Such a form exists over an infinite field. Multiplication by $l$ is
	injective on $k[X,Y]/(f_e)$, so $x$ is a nonzerodivisor and its
	initial form is regular. Furthermore
$	\operatorname{gr}_{\m}R/(l)\cong k[Z]/(Z^e).$
	Indeed restricting $f_e$ to the line $l=0$ gives a nonzero scalar
	multiple of $Z^e$. Regularity of the initial form makes the
	$x$-adic quotient filtration strict: if $xu\in\m^j$ then
	$u\in\m^{j-1}$. Thus the displayed quotient is the associated
	graded ring of $R/xR$, and $\len(R/xR)=e$.
	In degree $e$, the same graded quotient vanishes, yielding
	$\m^e=x\m^{e-1}+\m^{e+1}$. Applying Nakayama to
	$\m^e/x\m^{e-1}$ gives $\m^e=x\m^{e-1}$.
	Set $r=e-1$. Equation $
	(\ast)$ gives
	\[
	\len(R/\m^r)=\sum_{j=0}^{r-1}(j+1)=\frac{e(e-1)}2.
	\]
	Since $x$ is regular, successive quotients of $R/x^rR$ have
	length $\len(R/xR)=e$, so
	$\len(R/x^rR)=re=e(e-1)$. Both hypotheses of
	Lemma~\ref{lem:blowup} hold, proving the theorem for infinite $k$.
	Nothing in this argument uses a coefficient field or derivatives.
	To remove the infinitude assumption, form the faithfully flat local
	extension
	\(
	R'=R[Z]_{\m R[Z]}\) and \(k'=k(Z).
	\)
	It is again a one-dimensional local hypersurface domain, with
	regular ambient ring $S[Z]_{\n S[Z]}$, and its residue field is
	infinite. Its associated graded ring is the scalar extension of
	$	(\ast)$, so its multiplicity is unchanged.
	If $B=\overline R$, then $B'=B\otimes_RR'$ is finite, birational,
	and normal: it is a localization of the normal polynomial ring
	$B[Z]$. Thus it is the normalization of $R'$.
	
	For a finite birational extension, $\Hom_R(B,R)$ identifies with
	$R:B$ by multiplication in the fraction field. Hom from a finitely
	presented module commutes with flat base change. Hence
$(R:B)R'=R':B'.$
	Faithful flatness detects containment in every $\m^j$, since
	$\m'^j=\m^jR'$. Apply the already proved infinite-field argument
	to $R'$ and descend its two conductor containments to $R$.
\end{proof}

\section{The reflexivity problem of ideals of colength 2}

In this section, we explore Question \ref{q1}, motivated by the significance of Gorenstein property in Proposition \ref{37}.

\subsection{Homology of ideals of colength 2}
 For example,
 the conductor of $R:=k[[t^5,t^6,t^7]]$ is $\fm^2$  (see \cite[7.12]{dao}), but its multiplicity is $5$, and so  $\ord(\mathfrak{C})= e_\fm(R)-3$. It is well-known that $R$ is nearly Gorenstein, but not far-flung Gorenstein. Since type of the ring is two, one may
 inspired with the example and ask:
 \begin{question}
 	Let $(R,\fm)=k[[X,Y,Z]]/ I$ where $I\subseteq \fn^2$ be a  local domain of dimension $1$ and  $\mathfrak{C}$ is a power of $\fm$. 
When is $\ord(\mathfrak{C})= e_\fm(R)-(type(R)+1)$?  	\end{question}

\begin{example} Let
$R:=k[[t^3,t^4,t^5]]$. The following holds:
\begin{enumerate}
	\item [i)]  There is an ideal of colength two which is not reflexive, but the
	colength one ideals are reflexive.
		\item [ii)] Type of the ring is two, $R$ is   far-flung Gorenstein and nearly Gorenstein with $e_\fm(R)=3$.
		\item [iii)] The conductor is $\fm$. In particular, $\ord(\mathfrak{C})\neq e_\fm(R)-(type(R)+1)$.
\end{enumerate}
\end{example}

\begin{proof}
 Recall that $\omega_R=(t^3,t^4)$ is not reflexive, and that $\fm t^5\subseteq\omega_R$. This gives the exact sequence
$0\to \omega_R\to\fm\to k\to 0$ which induces the exact sequence $0\to\fm/ \omega_R\to R/ \omega_R\to R/\fm\to 0$ with $\fm/ \omega_R\cong k$. Since $\mathfrak{m}/\omega_R \cong k$, we deduce that
$\ell(R/\omega_R)=2$. The rest should be clear.
\end{proof}

The above example is almost Gorenstein and  of minimal multiplicity.

\begin{observation}\label{rem1}
Let $(R,\fm)$ be a $1$-dimensional complete almost Gorenstein domain, of minimal multiplicity.  Suppose any colength two ideal is  reflexive {and trace}.
	Then $R$  is  hypersurface. 
\end{observation}

\begin{proof}
	We may assume that $k$ is infinite. Suppose $R$ is not Gorenstein. By \cite[1.4 + 1.5]{kob}, there is an exact sequence
	\(
	0 \to \omega_R \to \mathfrak{m} \to k \to 0,
	\)
	for some isomorphic choice of the canonical module. Thus, we have 
	$
	\ell(R/\omega_R) = 2.
	$
	Since $\omega_R$ is trace, it follows that $\operatorname{tr}(\omega_R) = \omega_R$. Then, we have 
	$
	R = \Hom_R(\omega_R, \omega_R) = \Hom_R(\omega_R, R).
	$
	In particular, since $\omega_R \ncong R \cong R^* = \omega_R^{**}$, it follows that $\omega_R$ is not reflexive, a contradiction. Hence, $R$ is Gorenstein. Since the ring has minimal multiplicity, there exists some $x$ such that $\mathfrak{m}^2 = x \mathfrak{m}$. Let $\overline{R} := R/xR$, which is Gorenstein with a square-zero maximal ideal $\overline{\mathfrak{m}}$. Thus, $\overline{\mathfrak{m}} = \operatorname{Soc}(\overline{R})$, which is principal and generated by some $y$. Therefore, $\mu(\mathfrak{m}) = 2$, and by Cohen's structure theorem, $R$ is   hypersurface.
\end{proof}

%The almost Gorenstein assumption is important, even we assume the far-flung  Gorenstein property:

\begin{example}
Let $R:=k[[t^4,t^6,t^7,t^9]]$. The ring is  almost Gorenstein  domain, of minimal multiplicity, but not far-flung  Gorenstein.  Also, any colength two trace ideal is reflexive. \end{example}\begin{proof}	 As $(\fm:_{Q(R)}\fm)=K[[t^2,t^3]]$ is Gorenstein, thanks to \cite[5.1]{Gt}, the ring is almost Gorenstein. As $\mu(\fm)=4=e_{\fm}(R)$ the ring is of minimal multiplicity. In view of \cite[6.1]{HHS}, the ring is nearly Gorenstein, and so $\tr(\omega_R)=\fm$.
Since $\mathfrak{C}\neq\tr(\omega_R)$, the ring is  not far-flung  Gorenstein and that $\ell(R/\mathfrak{C})= 2$.
According to \cite[6.3]{dao},
the only trace ideals are $\{\mathfrak{C},\fm,R\}$. By \cite[Ex. 12.13]{HS}, $\mathfrak{C}^\ast=\overline{R}$. Consequently,   $\mathfrak{C}^{\ast\ast}=\overline{R}^\ast=\mathfrak{C}$, and the claim follows.\end{proof}

\begin{proposition}
Let $(R,\fm)$ be a $1$-dimensional complete almost Gorenstein domain, of minimal multiplicity. Suppose $R$  is not hypersurface.   Then $0\to\omega_R\to\omega_R^{\ast\ast}   \to k\to 0$ is exact, and in particular, $\omega_R^{\ast\ast}\cong \fm $.
\end{proposition}

\begin{proof}
  By the argument of Observation \ref{rem1}, we assume in addition that $R$ is not Gorenstein.
In particular, $\omega_R$ is not reflexive. We apply duality to the exact sequence
$
0 \to \omega_R \to \mathfrak{m} \to k \to 0,
$
and obtain the exact sequence
$
0 = \Hom_R(k, R) \to \mathfrak{m}^* \to \omega_R^* \to \Ext^1_R(k, R).
$
So, we can find a vector subspace $X \subset \Ext^1_R(k, R)$ such that the sequence
$
0 \to \mathfrak{m}^* \to \omega_R^* \to X \to 0
$
is exact.  Taking another duality, we obtain the following exact sequence
$
0 \to \omega_R^{**} \to \mathfrak{m}^{**} \to k^n \to 0.
$
Here, $n \in \mathbb{N}_0$. 
Recall from the work of Eilenberg-MacLane that biduality is a functor. This yields the following commutative diagram$$\begin{CD}@.@. \quad 0
\\@.@.   \quad@AAA    \\0@>>>  \quad\omega_R^{\ast\ast} @> >>  \quad\fm ^{\ast\ast} @>>> \quad k^{n} @>>> 0\\
@.\delta_{\omega_R}@AAA \delta_{\fm}@AAA \exists g @AAA    \\
0@>>> \quad\omega_R @>>>  \quad\fm @>>> \quad k@>>>0.\\
@.  \quad\quad@AAA  \quad@AAA \\@.  \quad\quad0 @.\quad0
\end{CD}	$$

Now, using the $\ker-\coker$ sequence, we obtain the exact sequence
\[
0 = \ker(\delta_{\mathfrak{m}}) \to \ker(g) \to \coker(\delta_{\omega_R}) \to \coker(\delta_{\mathfrak{m}}) = 0.
\]
Recall that $\ker(g)$ is either zero or equal to $k$, because $k$ is simple. Since $\omega_R$ is not reflexive, $\coker(\delta_{\omega_R}) \neq 0$. Thus, $k = \ker(g) = \coker(\delta_{\omega_R})$. In other words, the  sequence
$
0 \to \omega_R \to \omega_R^{**} \to k \to 0
$
is exact. Again using the $\ker-\coker$ sequence, we obtain the isomorphism 
$
0 = \coker(\delta_{\mathfrak{m}}) \to \coker(g) \to 0.
$
But $\coker(g) = k^n / \operatorname{im}(g) = k^n$, so combining these results implies that $n = 0$, and thus $\omega_R^{**} \cong \mathfrak{m}$.
\end{proof}

\begin{notation}\label{para20240911c}
Let $F_1\stackrel{{\pi}}\lo F_0\to M\to 0$ be a free presentation for an $R$-module $M$. Applying the functor $(-)^*$, we get the exact sequence
$$
	0\to M^*\to F_0^*\stackrel{{\pi^*}}\lo F_1^*\to \coker(\pi^*)\to 0 \quad(\dagger)
$$We denote $\coker(\pi^*)$ by $\Tr M$; it is called the \emph{Auslander transpose of $M$}.
The Auslander transpose depends on the free presentation of $M$, however, it is unique up to free summand.
%Also, the $R$-modules $\Tr \Tr M$ and $M$ are isomorphic up to free summands.
\end{notation}
Recall from \cite{ab} that $M$ is q-torsionless provided \( 
\Ext_R^{[1,q]}(\operatorname{Tr}(M), R) = 0.
\) Recall that 2-torsionless is a reformulation of  reflexivity.
\begin{proposition}\label{hyp2}
Let $(R,\fm)$ be a $1$-dimensional domain of minimal multiplicity. Suppose there exists a self-dual ideal of colength two that is $q$-torsionless with $q > 2$. Then $R$ is a hypersurface.
\end{proposition}

\begin{proof}
Let $I$ be such an ideal of colength two. Recall from \cite[Proposition 3.10]{dao} that $I$ is either integrally closed or principal. If $I$ is principal, then $R$ must be a hypersurface. Indeed, considering the short exact sequence  
$
	0 \to \mathfrak{m}/xR \to R/xR \to R/\mathfrak{m} \to 0,
$ we deduce that $\ell(\mathfrak{m}/xR) = 1$. Suppose $\mathfrak{m}/xR$ is generated by $y + xR$ for some $y \in \mathfrak{m}$, then $\mathfrak{m} = (x, y)$. Since $\mu(\mathfrak{m}) = \dim R + 1$, it follows that $R$ is a hypersurface.  
Thus, without loss of generality, we assume that $I$ is integrally closed. Since $I$ is q-torsionless, we have  
\( 
	\Ext_R^q(\operatorname{Tr}(I), R) = 0.
\)
 As $R$ is a domain, $\operatorname{Tr}(I)$ is locally free over the punctured spectrum. Let $(-)^\vee:=\Hom(-,E_R(k))$ be the Matlis dual. By \cite[Lemma 3.5(2)]{ta}, we obtain \[
	\Tor_{q-2}^R(I, R/\omega_R) = \Tor_{q-2}^R(I^\ast, R/\omega_R) \stackrel{ (\dagger)}= \Tor_q^R(\operatorname{Tr}(I), R/\omega_R) = \Tor_{q-1}^R(\operatorname{Tr}(I), \omega_R ) \cong \Ext_R^q(\operatorname{Tr}(I), R)^{\vee} = 0.
	\]  
	Since $R$ is $1$-dimensional, $I$ is integrally closed, $\fm$-primary, and $\Tor_{q-1}^R(R/I, R/\omega_R) = 0$, it follows from \cite[3.3]{corso} that  
	\( 
	\operatorname{pd}(R/\omega_R) \leq q-2.
	\)
In view of the Auslander-Buchsbaum formula, $\omega_R$ is free. Thus, $R$ is Gorenstein. As noted earlier, this implies that $R$ is a hypersurface.  
\end{proof}

\begin{corollary}\label{hyp2c}
	Let $(R,\fm)$ be a $1$-dimensional domain of minimal multiplicity. Suppose there exists a totally reflexive ideal  $I$ of colength two. Then $R$ is a hypersurface.
\end{corollary}

\begin{proof}
Without loss of generality, we assume that $I$ is integrally closed. Since $I^\ast$ is totally reflexive as well as  $I$, and  so $\Ext_R^+(\Tr(I^\ast)),R)=0$.
	This shows, for any $q>3$, that 
\[
	\Tor_{q-3}^R(I,  \omega_R) = \Tor_{q-3}^R(I^{\ast\ast}, \omega_R) \stackrel{ (\dagger)}= \Tor_{q-1}^R(\operatorname{Tr}(I^\ast), \omega_R ) \cong \Ext_R^q(\operatorname{Tr}(I^\ast), R)^{\vee} = 0.
	\] So,  $\Tor_{q-2}^R(R/I,  \omega_R) =0$.
	This forces the desired claim. 
\end{proof}

\subsection{A criterion in an Artinian quotient}
Let $(R,\m,k)$ be one-dimensional Cohen--Macaulay. Every ideal $I$
with $\len(R/I)=2$ contains $\m^2$, and $\m/I$ has length one.
Consequently all such ideals are parametrized by hyperplanes
\[
H\subseteq\m/\m^2,\qquad I_H=\text{inverse image of }H.
\]
This is a bijection, since every subspace of $\m/\m^2$ lifts to an
ideal. The parameter set is the projective space
$\mathbb P((\m/\m^2)^*)$ of nonzero linear functionals up to scalar.
Choose a nonzerodivisor $a\in\m^2$ and fix the Artinian ring
\[
A=R/a\m^2,\qquad \alpha=a+a\m^2,
\qquad P=\alpha A.
\]
Such an $a$ exists since $\m^2$ is $\m$-primary and $R$ has depth
one. For an ideal $I_H$, let $L_H=I_H/a\m^2$ and define the colon
entirely inside this one fixed Artinian ring by
\[
J_H=(P:_A L_H)=\{u\in A:uL_H\subseteq P\}.
\]
\begin{observation}\label{thm:colengthfinite}
	For every hyperplane $H$ one has
	\begin{enumerate}
	\item[(i)]
		$I_H\text{ is reflexive}
	\quad\Longleftrightarrow\quad(P:_A J_H)=L_H,$
			\item[(ii)]
		$I_H\text{ is trace}
	\quad\Longleftrightarrow\quad J_HL_H=\alpha L_H.
	$
	\end{enumerate}

\end{observation}
\begin{proof}
	Let $I=I_H$ and put $B_0=(aR):_R I$. Since $a\in I$ is regular,
	localizing at nonzerodivisors identifies any map $I\to R$ with
	multiplication by an element of $\Q(R)$. Its multiplier lies in
	$a^{-1}R$, and more precisely
	\[
	I^*=a^{-1}B_0,\qquad
	\tr_R(I)=a^{-1}B_0I,\qquad
	I^{**}=(aR):_R B_0.
	\]
	For the last equality, $R\subseteq I^*$, so a bidual multiplier
	lies in $R$; its condition is exactly $rB_0\subseteq aR$.
	The natural evaluation is the inclusion in this colon.
The common ideal $a\m^2$ is contained in $aI$, in $aR$, and in
	$B_0$. Passing to $A$ therefore gives
	\[
	B_0/a\m^2=J_H,
	\quad ((aR):_R B_0)/a\m^2=(P:_A J_H),
	\]
	and
	\[
	(B_0I)/a\m^2=J_HL_H,
	\qquad (aI)/a\m^2=\alpha L_H.
	\]
	Reflexivity is equality $I=(aR):B_0$, giving
(i). Trace is equality $B_0I=aI$, giving
(ii).
\end{proof}

\begin{corollary}\label{cor:artinreflexive}
	With $\widetilde A=R/aR$ and $\widetilde L=I/aR$, a colength-two
	ideal is reflexive exactly when
	\(
	\widetilde L=\Ann_{\widetilde A}\Ann_{\widetilde A}\widetilde L.
	\)
	In particular every colength-two ideal of a one-dimensional
	Gorenstein local ring is reflexive\footnote{this may follow by Auslander-Bridger formula.}.
\end{corollary}

\begin{observation}\label{thm:ordinary}Let $B=k[[t]]$ and, for $e\ge3$, let
	\(
	R_e=k+t^eB=k[[t^e,t^{e+1},\ldots,t^{2e-1}]]\) and \(\m=t^eB\).
	Then $\m^2=t^{2e}B$, so a colength-two ideal is specified by a
	hyperplane in the span of $t^e,\ldots,t^{2e-1}$.
	For any $e\ge3$, exactly one colength-two ideal of
	$R_e$ is reflexive, namely
	\(
	I_0=t^{e+1}B.
	\)
	No colength-two ideal is trace. More precisely, every such ideal has
	trace $\m$, and every $I\ne I_0$ has $I^{**}=\m$.
\end{observation}
\begin{proof}
	If the hyperplane has no element with nonzero $t^e$ coefficient,
	it is the span of $t^{e+1},\ldots,t^{2e-1}$, so $I=I_0$.
	For $I_0$ one checks
	\[
	R_e:I_0=t^{-1}B,\qquad
	R_e:(t^{-1}B)=t^{e+1}B.
	\]
	Here a multiplier of valuation at least $-1$ works in the first
	colon. A smaller valuation either gives a negative exponent or a
	forbidden exponent between $1$ and $e-1$ when multiplying suitable
	consecutive monomials of $I_0$. The second colon follows from
	$R_e:B=t^eB$. Thus $I_0$ is reflexive and
	$\tr(I_0)=t^{e+1}B\,t^{-1}B=\m$.
	
	Now let $I\ne I_0$. It contains $f$ of valuation $e$, and it also
	contains $g$ of valuation $j$ with $e<j<2e$. To see the latter,
	its hyperplane has dimension $e-1\ge2$; subtracting multiples of
	the $t^e$ coefficient leaves a nonzero class of higher valuation.
	Clearly $B\subseteq R_e:I$ since $BI\subseteq t^eB$.
	No multiplier $z$ of negative valuation can belong to this colon.
	Writing $v(z)=-h$, if $1\le h<e$, then $zf$ has valuation
	$e-h\in\{1,\ldots,e-1\}$; if $h>e$, its valuation is negative;
	and if $h=e$, then $zg$ has  valuation $j-e$.
	Therefore $R_e:I=B$ and $I^{**}=R_e:B=\m$. Also $fB=t^eB$,
	so $\tr(I)=IB=\m$. Since $\len(R_e/\m)=1$, no colength-two
	ideal equals its trace, and precisely $I_0$ equals its bidual.
\end{proof}

\section{Prime characteristic methods}

Rings in this section are of prime characteristic $p$. \begin{discussion}
\begin{enumerate}\label{dp}
	\item [i)] The absolute integral closure of an integral domain \( R \), denoted by \( R^+ \), is the integral closure of \( R \) inside an algebraic closure of its fraction field. For a reduced ring $R$, we define $R^+:=\oplus_{\fp\in \Ass(R)}(R/\fp)^+$. Let $R_{\red}:=\frac{R}{\nil(R)}$. For a general ring, we set $R^+:=(R_{\red})^+$. 
	
\item [ii)] Recall  that the  perfect closure $R^\infty$  of  $R$ is defined by  $R^\infty:=\varinjlim\xymatrix{(R\ar[r]^{\F}&R \ar[r]^{\F}&\ldots),}$
	where $\F:R\to R$ is the Frobenius map.
	Since $\F^{\gg 0}$
	kills nilpotent elements, $R^\infty$ is reduced and exists uniquely. This sometimes is called the minimal perfect algebra.
 In fact, $R^\infty$ is defined
	by adjoining to $R_{\red}$ all $p$-power roots of elements of $R_{\red}$.

\item [iii)]  $R^{\frac{1}{p^n}}:=\{x\in
R^\infty:x^{p^n}\in R \textmd{ for some
}n\in\mathbb{N}_{0}\}.$ When $R$ is reduced, this may denoted by $\up{\F _n}( R)$.
\end{enumerate}\end{discussion}Each iteration $\F_n$ of $\F$ defines a new $R$-module structure on the set $R$, and this $R$-module is denoted by $\up{\F_n}R$, where $a\cdot b = a^{p^{n}}b$ for $a, b \in R$. 
Recall that $R$ is said to be $\F$-\emph{finite} if  for some (or
equivalently, all) $n$, the algebra   $\up{\F_n}R$  is   finitely generated as an $R$-module.
\subsection{Is absolute integral closure  free?} 
In this  subsection, we address Question \ref{q4} by presenting the following observation.

\begin{proposition}\label{fr}
	There exists a $\DVR$ of prime characteristic \( p \) for which
\( 
	\pd_R(R^+)=1.
\)
	In particular, \( R^+ \) is not free as an \( R \)-module.
\end{proposition}

\begin{proof} Suppose, toward a contradiction, that
\( \pd_R(R^+)\neq 1.
\) Since \( R \) is 1-dimensional and regular, every \( R \)-module has projective dimension at most one. Therefore, if \(\pd_R(R^+)\neq 1\), then \(\pd_R(R^+)=0\), meaning that \( R^+ \) is projective. By the local property of \( R \), every projective module is free, so \( R^+ \) must be free.
Moreover, over a $\DVR$ every submodule of a free module is again free. In particular, applying this to the inclusion
\(
	R^{\frac{1}{p^n}} \subseteq R^\infty \subseteq R^+,
	\)
we deduce that \( R^{\frac{1}{p^n}} \) is free, and would have a nontrivial direct sum decomposition. This conclusion contradicts \cite[Corollary 5.3]{smith}\footnote{They also construct a DVR $R$ so that $\Hom(R^{1/p}, R)=0$. In particular, $R^{1/p}$ need not be reflexive. Despite this, see next section, where we investigates the reflexivity problem in 1-dimensional case.}. Hence, the assumption must be false, and we must have \(\pd_R(R^+)=1\).
\end{proof}

However, the above ring is not complete. We show if the above conclusion  is failed for its completion, then $R^{\frac{1}{p^n}}$ is not reflexive, see the following subsection.

\subsection{Is $R^{\frac{1}{p^n}}$ reflexive?}
Thanks to \cite[final comment]{dao},  when $R$ is one-dimensional complete local domain with algebraically closed residue field $\overline{K}$ and for large $q$, the $R$-module $R^{\frac{1}{p^n}}$  is reflexive. 
\begin{proposition}\label{1ref}
Let $(R,\fm,{K})$ be  1-dimensional analytically unramified ring and $n\gg 0$. Then  	$R^{\frac{1}{p^n}} $ is reflexive provided it is finitely generated.
	\end{proposition}
Recall that an analytically unramified ring is a local ring whose completion is reduced.
\begin{proof}
First, assume $\overline{K}={K}$, and we reduce to the case $(R,\fm,{K})$ is complete. Let $S$ be the $\fm$-adic completion of $R$.  Also, assume $S$ is domain.
The ring extension $(S,\fn)$ of $(R,\fm)$ such that $\fm S=\fn$, $S$ is $\F$-finite, $S$ is faithfully flat over $R$.   For an $R$-module $L$, it follows that
\begin{enumerate}
	\item [i)]  $(L\otimes_R \up{\F_n }R)\otimes_R S\cong (L\otimes_RS)\otimes_S\up{\F_n }S$,	\item[ii)] the map  $\up{\F_n}R\stackrel{\cong}\lo R\otimes_R\up{\F_n}R\lo  S\otimes_R\up{\F_n}R\stackrel{\cong}\lo \up{\F_n }S$  is faithfully flat via base change theorem. 
\end{enumerate}
 Now, we bring  the following for a finitely generated module $M$:
\begin{enumerate}
	\item [iii)]  $\Tr_S(M\otimes_R S)= \Tr_R(M)\otimes_R S$,
	\item [iv)] $\Ext^i_R(M, R)\otimes_R S=\Ext_S^i(M\otimes_R S, S)$,	\item [v)] the sequence  $0\to\Ext^1_R(\Tr_R(M), R)\to M\to M^{\ast\ast}\to \Ext^2_R(\Tr_R(M), R)\to 0$ is exact.
\end{enumerate}
From these and by replacing $R$ with $S$, we may and do assume that $R$ is complete, and the claim follows by \cite[final comment]{dao}.
Suppose
$R$ is complete. We want to drop the extra assumptions, e.g., $\overline{K}={K}$. Since $R$ is reduced, we have  $\widehat{R}\otimes_R \up{\F_n }R\cong \up{\F_n }\widehat{R}$. This implies that ${K}$ is a perfect filed. Instead using this, let us use Cohen's structure theorem and write $\widehat{R}\cong K[[x_1,\cdots, x_m]]/I$ for some ideal $I$, we can pick $\tilde{S}:=\bar{K}[[x_1,\cdots,x_m]]/I\bar{K}[[x_1,\cdots, x_m]]$, where $\bar{K}$ denotes the algebraic closure of $K$.  We are going to replace $S$ with $\tilde{S}$, via  repeating the previous argument, and remark that the ring new ring $S$ may has zero-divisor, but it has no embedded prime divisor as it is Cohen-Macaulay. Due to Discussion \ref{dp} we may assume $\tilde{S}$ is reduced, and by Chinese remainder theorem $\tilde{S}=\oplus_{\fp \in\Ass(\tilde{S})} S/\fp$. Recall by the previous argument that $(\tilde{S}/ \fp)^{\frac{1}{p^n}} $ is reflexive as an $\tilde{S}/\fp$-module  for some large enough $n$. Then 
$(\tilde{S})^{\frac{1}{p^n}}=\oplus (\tilde{S}/ \fp)^{\frac{1}{p^n}}$ is reflexive as an $(\tilde{S})_{\red}$-module, and so an $\tilde{S}$-module.
\end{proof}
\iffalse
\begin{question}\label{q2}
	Let $R$ be  Cohen-Macaulay  and of prime characteristic with canonical module and $n\gg 0$.\begin{itemize}
		\item[$(1)$] (See \cite[Page 167]{mar}) Suppose $\Ext^{j}_R(\up{\varphi_n}R,R)=0$ for some $j>\dim R$.  Is $R$ Gorenstein? 
		\item[$(2) $]  (See \cite[2.12]{l}) Suppose $\Tor^R_i(\up{\varphi_n}R,\omega_R)=0$ for some $i>0$. Is $R$ Gorenstein?
		
	\end{itemize}
\end{question}

\begin{corollary}
Both Questions are true over F-finite Arf domains.
\end{corollary}

\begin{proof}
	Since these questions are equivalent, we prove only 2).
	We first show that 
	$(\up{\F_n}R)$ is reflexive for any large $n$.
	In view of Observation  \ref{1ref} the claim is clear
in dimension one, so:
	\begin{enumerate}[(a)]
		\item $(\up{\F_n}R)_{\fp}$ is reflexive for all $\fp$ with  $\depth R_{\fp} \leq 1$, and
		\item $\depth (\up{\F_n}R)_{\fp} \geq 2$ for all $\fp$ with  $\depth R_{\fp} \geq 2$.
	\end{enumerate}From these, $(\up{\F_n}R)$ is reflexive for any large $n$.
\end{proof}\fi
\begin{remark}\label{dpD}
\begin{enumerate}
\item [i)] Suppose
$R$ is d-dimensional, $\F$-finite and Gorenstein. Then 	$R^{\frac{1}{p^n}} $ is (totally) reflexive for all $n$.

\item [ii)] Suppose $R$ is zero dimensional and let $n$ be so that $\fm^{p^{n}}=0$. Then $R$ is  Gorenstein  iff $\up{\F_n}R$ is reflexive. 

\item [iii)] There are rings (not necessarily artinian)  for which $\up{\F_n}R$ is not reflexive for all $n>0$.

\item [iv)]
 Suppose $R$ is    $\F$-finite and $\up{\F_n}R$ is  reflexive. Let
$R \to S$ be  weakly \'{e}tale. Then $\up{\F_n}S$ is  reflexive.

\item [v)] Suppose $R$ is    $\F$-finite and $\up{\F_n}R$ is  reflexive. Let
$R \to S$ be  weakly \'{e}tale in co-dimension 1 and $R$ be $(S_2)$. Then $\up{\F_n}S$ is  reflexive.

\end{enumerate}
\end{remark}

\begin{proof}
i):  By Auslander-Bridger formula, 	$R^{\frac{1}{p^n}} $ is (totally) reflexive for all $n$.

 ii):   Suppose $R$ is not Gorenstein.  
		Since $\fm \cdot(\up{\F_n}R)=\fm^{[p^n]}\subseteq \fm^{p^n}=0$, 
		it follows $\up{\F_n}R$
		is a $\mathbb{F}_p$-vector space of dimension say $m$. Then $(\up{\F_n}R)^{\ast\ast}$ is a $\mathbb{F}_p$-vector space of dimension $m\times \textit{type}^2(R)$. Since $m>0$ and $\textit{type}(R)>1$, we conclude that $\up{\F_n}R$ is not reflexive.
		The reverse implication is also follows from item i).  

iii): This is similar to \cite[Example 3.7]{asg}. A typical example is $R:=\frac{\mathbb{F}_p[[X,Y,Z]]}{(X,Y)(X,Y,Z)}$. We remark that $R$ is one-dimensional
and of depth zero. The straightforward modification is leave to the  reader.
		
iv):  Since $R \to S$ is  weakly \'{e}tale, then $S^{1/p^n}\simeq S\otimes_RR^{1/p^n}$. It remains to recall that any flat extension of
		a reflexive module, is again reflexive.  

v):  Since the rings satisfy $(S_2)$, reflexivity reduces to checking the condition in codimension one. Now, apply item iv). 
\end{proof}
Let $(R,\fm)$ be $(S_1)$, $\F$-finite and positive dimension. One may  conjecture that $\up{\F_n}R$ is  reflexive for some $n$. In what follows, we present both positive and negative answers.
\begin{corollary}\label{55}
	Suppose in addition to the previous paragraph,  $R$ is $(G_1)$. Then  $\up{\F_n}R$ is reflexive.
\end{corollary}

The following shows the $(G_1)$ condition is really needed in Corollary \ref{55}
\begin{example}\label{thm:frobcounter}
Let $k=\mathbb F_p$,	\(
R=\frac{k[[z,x,y]]}{(x,y)^2}\), and
\(S=k[[z]].
\)
Then $R$ is complete, $F$-finite, one-dimensional, and Cohen--Macaulay,
and it satisfies $(S_1)$. Nevertheless, for every $n\ge1$,
$F_*^nR$ is not reflexive. With $q=p^n$, as $S$-modules,
\[
F_*^nR\cong S^{3q},\qquad
(F_*^nR)^{**}\cong S^{12q},
\]
and the evaluation map is injective with cokernel $S^{9q}$.
\end{example}
\begin{proof}
We start by denoting 
$N:=(x,y)R.
$
Since $(x,y)^2=0$, we have
$N^2=0$, $N=Sx\oplus Sy,$
and
$R=S\oplus Sx\oplus Sy$
as $S$-modules. In particular, $R$ is free of rank three over $S$.
Moreover, $z$ is a nonzerodivisor on $R$. Since $\dim R=1$,
$
\depth R=1=\dim R,
$ so $R$ is Cohen--Macaulay. In particular, $R$ satisfies $(S_1)$.
Now let
$q=p^n.$
Recall that the $R$-module $F_*^nR$ has the twisted scalar action
$r\cdot m=r^qm.$
If $u\in N$, then $u^2=0$, and hence $u^q=0$ since $q\geq2$.
Thus
$N\cdot F_*^nR=0.$
Consequently, the scalar action on $F_*^nR$ factors through
$R/N=S.$
More precisely, the resulting $S$-action is the Frobenius-twisted
action
$s\cdot m=s^qm.$
Therefore
\(
F_*^nR
\cong F_*^nS\oplus F_*^nSx\oplus F_*^nSy
\)
as $S$-modules.
Since $k=\mathbb F_p$, Frobenius on $k$ is an automorphism. Every
$f(z)\in S$ has a unique expression
\[
f(z)=\sum_{i=0}^{q-1}z^i g_i(z)^q,
\qquad g_i(z)\in S.
\]
Indeed, group the exponents of $f$ according to their residue classes
modulo $q$ and take the unique $q$th roots of the coefficients in
$k$. Hence
$F_*^nS\cong S^q$
as $S$-modules, and therefore
${F_*^nR\cong S^{3q}.}$
This also shows directly that $F_*R$ is a finite $R$-module, so $R$
is $F$-finite.
Next, since
$\operatorname{Ann}_R(N)=N,$
any $R$-linear map from an $R$-module on which $N$ acts trivially to
$R$ has image contained in $N$. Thus, for any finite free $S$-module
$M$ on which $N$ acts trivially,
\(
M^*=\operatorname{Hom}_R(M,R)
=\operatorname{Hom}_S(M,N).
\)
Consequently,
\(
M^{**}
=
\operatorname{Hom}_S\bigl(\operatorname{Hom}_S(M,N),N\bigr).
\)
For $M=S$, since
\(
N\cong S^2,
\)
we obtain
\(
S^*
=\operatorname{Hom}_S(S,N)
\cong N\cong S^2
\)
and hence
\(
S^{**}
\cong\operatorname{Hom}_S(S^2,S^2)
=\operatorname{End}_S(N)
\cong M_2(S)\cong S^4.
\)The evaluation map is
$
\delta:S\to\operatorname{End}_S(N),$
$s\mapsto s\,\operatorname{id}_N.$
After identifying $N\cong S^2$, this is
\[
s\longmapsto
\begin{pmatrix}
	s&0\\
	0&s
\end{pmatrix}.
\]
It is injective because $S=k[[z]]$ is a domain.
Moreover, the image consists of the scalar matrices. The map
\[
\lambda:M_2(S)\longrightarrow S,
\qquad
\begin{pmatrix}a&b\\c&d\end{pmatrix}\longmapsto a
\]
satisfies
$\lambda(sI_2)=s.$
Hence the inclusion
\(
S\hookrightarrow M_2(S)\), \(s\mapsto sI_2,
\)
splits. Therefore
$\operatorname{coker}(\delta)\cong S^3.$ Finally,
$F_*^nR\cong S^{3q}.$
Taking duals gives
$(F_*^nR)^*
\cong S^{6q},$
and hence
$(F_*^nR)^{**}\cong S^{12q}.$
The evaluation map
$F_*^nR\to(F_*^nR)^{**}$
is the direct sum of $3q$ copies of the evaluation map
$S\to S^4.$
Thus it is injective and
$\operatorname{coker}
\bigl(F_*^nR\to(F_*^nR)^{**}\bigr)
\cong S^{9q}.
$
In particular, the evaluation map is not surjective, so $F_*^nR$
is not reflexive for every $n\geq1$.
\end{proof}

The following connects subsection 5.1 to 5.2.
\begin{proposition}
Let $K$ be a field which is not $\F$-finite, and let $R:=K[[X]]$. Then either $\up{\F_n}R$ is not reflexive or $R^+$ is not free.
\end{proposition}

\begin{proof} Suppose \( R^+ \) is free. Thanks to the proof of Proposition \ref{fr}, \( \up{\F_n}R \) is also free.  
Recall that a complete local ring \( (R, \fm, K) \) is \( \F \)-finite if and only if \( K \) is \( \F \)-finite. Hence, \( R \) is not \( \F \)-finite. In particular, \( \bigoplus_{\aleph_0} R \) is a directed summand of \( \up{\F_n}R \).  
	Here, we apply the theory of slender modules. For definitions and basic properties, see \cite[Chapter III]{EM02}. Recall from \cite[Theorem III.2.9]{EM02} that a $\PID$ is slender if and only if it is not complete in the \( R \)-adic topology. Since \( \dim(R) = 1 \) and \( R \) is local, the \( R \)-adic topology coincides with the \( \fm \)-adic topology. Thus, \( R \) is not slender.  
	This shows that
	\[
	(\bigoplus_{\aleph_0} R)^{\ast\ast} = (\prod_{\aleph_0} R)^{\ast} \neq \bigoplus_{\aleph_0} R.
	\] Since a direct summand of a reflexive module is again reflexive, we deduce that \( \up{\F_n} R \) is not reflexive.
	\end{proof}

\end{document}